\documentclass[runningheads]{llncs}

\usepackage[T1]{fontenc}
\usepackage{afterpage}
\usepackage{graphicx}
\usepackage{tikz}
\usetikzlibrary{calc,arrows.meta}
\usepackage{subcaption}
\usepackage[section]{placeins}

\usepackage{amsmath,amssymb,amsfonts,mathtools}
\usepackage{algorithm}
\usepackage{algpseudocode}

\newcommand{\wrap}{\operatorname{wrap}_\pi}

\begin{document}
\title{Topology, noise, and parallel updates in a model of circular opinion dynamics}
\titlerunning{Order parameters for opinion dynamics}
%
\author{
Wioletta M. Ruszel\textsuperscript{1}
\and
Cristian Spitoni\textsuperscript{2}
}
\authorrunning{W.M. Ruszel and C. Spitoni}
\institute{
Mathematical Institute, Utrecht University, Utrecht, The Netherlands
}

\maketitle

\begingroup
\renewcommand\thefootnote{\arabic{footnote}}
\footnotetext[1]{\texttt{w.m.ruszel@uu.nl}}
\footnotetext[2]{\texttt{c.spitoni@uu.nl}}
\endgroup

\begin{abstract}
We study a circular opinion dynamics model with local midpoint interactions, extended to allow parallel updates of multiple sites. 
On a ring, the dynamics admits twisted states associated with integer winding numbers. 
We investigate how bi-modal noise, which drives opinions toward two antipodal directions, affects these configurations. 
Numerically, we find that noise both destabilizes winding states and induces a spin--flop regime, characterized by macroscopic switching between preferred orientations. 
We introduce order parameters that distinguish topological trapping from symmetry breaking, providing a simple macroscopic description of the dynamics.
\end{abstract}
\section{Introduction}
Opinion dynamics and synchronization models are often studied through
minimal interaction rules that isolate the mechanisms responsible for
collective behavior \cite{Lorenz2007}. While classical opinion models
typically assume opinions lie on a linear scale, many systems involve
variables that are intrinsically periodic, including orientations,
phases, cyclic preferences, or states defined modulo $2\pi$. In such
settings, circular state spaces arise naturally and introduce
topological features that have no analogue in Euclidean opinion
dynamics. Similar circular variables play a central role in the
Kuramoto model and related synchronization frameworks, where winding
numbers and twisted states organize the collective dynamics
\cite{Acebron2005,Strogatz2000}. The present work adopts this
perspective and studies a minimal circular averaging process designed
to isolate how topology, local compromise, stochastic perturbations,
and synchrony interact. The path and ring are chosen as the simplest
interaction topologies that allow one to compare systems with and
without topological winding. Likewise, the midpoint rule is the
natural circular analogue of arithmetic averaging in classical
compromise models, as it moves interacting agents toward the midpoint
of the shortest arc connecting their opinions. Finally, the bi-modal
perturbation can be interpreted as a stylized representation of
competing external influences favoring two opposite orientations. Our
goal is therefore not to model a specific empirical social process,
but rather to understand the dynamical consequences of these
ingredients in the simplest setting where they can be analyzed
systematically.

In a recent companion work \cite{Brockhaus2026}, we introduced the \texttt{ACCA} (Asynchronous Continuous-state Cellular Automaton) midpoint dynamics as a mathematically tractable model of circular opinion formation, designed to isolate how nonlinear geometry and graph topology interact in local averaging processes. In this framework, neighboring agents repeatedly move to the midpoint along the shortest arc on the circle, producing a dynamics that is locally contractive but globally constrained by the topology of the interaction graph. We established a rigorous theoretical foundation for this model by proving that, for every finite system size, the process converges almost surely to consensus under both \emph{open} and \emph{periodic} boundary conditions. More importantly, however, we showed that \emph{periodic} boundary conditions fundamentally reorganize the transient dynamics through the emergence of topological winding sectors. In the ring geometry, wrapped local opinion differences define an integer-valued winding number, which partitions the configuration space into distinct sectors corresponding to different global topological profiles. We proved that this winding number can change only through explicitly characterizable branch-crossing events, thereby identifying the unique mechanism by which the system may escape from one topological sector to another. Starting from generic disordered initial conditions, the local midpoint rule rapidly suppresses large angular discrepancies, causing branch-crossings to become increasingly rare and effectively freezing the system for long time intervals within fixed winding sectors. Within these sectors, we developed a lifted representation, allowing us to rigorously transform the nonlinear circular dynamics into an exact Euclidean midpoint process. This construction yielded strict contraction results and provided a mathematical explanation for the emergence of long-lived twisted transient configurations: while these winding profiles are not true equilibria, the system is dynamically trapped near them for extended times before rare fluctuations eventually drive relaxation toward consensus. In this way, the previous article established a minimal and fully rigorous mechanism through which circular geometry alone induces \emph{metastability}, topological trapping, and slow relaxation, despite consensus remaining the only absorbing state. 

The present work builds directly on this mathematical framework by investigating how these previously characterized trapping mechanisms are altered when additional ingredients, namely \emph{bi-modal} stochastic perturbations and partial \emph{parallelization}, are introduced. Rather than re-establishing the deterministic topological structure, our goal here is to understand how noise and synchrony destabilize winding sectors, compete with midpoint contraction, and generate qualitatively new macroscopic phenomena such as symmetry breaking, winding destabilization, and \emph{spin–flop} collective dynamics.

We consider indeed  a noisy extension of the \texttt{ACCA} dynamics in which opinions are occasionally pushed toward two antipodal directions. These bi-modal perturbations generate a new collective phenomenon: the system alternates between two coherent macroscopic orientations, producing long-lived \emph{spin--flop} transitions \cite{collet2016s}.

In addition, we introduce a generalized version of the dynamics with \emph{parallel updates}, in which multiple disjoint interactions and stochastic perturbations act simultaneously on subsets of agents. This extension allows us to interpolate between asynchronous and partially synchronous regimes, and to investigate how the degree of parallelism affects the formation and stability of coherent structures.

Our aim is to characterize these regimes through suitable macroscopic observables. We introduce the models in Section \ref{sec:model} and the order parameters that distinguish winding trapping from antipodal switching in Section \ref{sec:order}. They will be used in Section \ref{sec:results} in the simulations to study the interplay between topological constraints, midpoint averaging, bi-modal noise, and parallel updates. Finally, we discuss future developments in Section \ref{sec:discussion}.
\section{Model and notation}
\label{sec:model}

This section uses the notation introduced in \cite{Brockhaus2026}. We consider a system of $N$ agents located on the vertices
$V_N:=\{1,\dots,N\}$ of a one–dimensional graph with edge set $\mathcal{E}$.
Each agent carries an opinion represented by an angle on the circle,
so that the configuration at time $t$ is $\theta_t=(\theta_t(1),\dots,\theta_t(N))\in\Omega_N:=[-\pi,\pi)^N$. Throughout the paper angles are interpreted modulo $2\pi$.
For $x\in\mathbb R$, we define the {\it wrapped difference}
\begin{equation}
\label{e:wrap}
\wrap(x):=x-2\pi\Big\lfloor\frac{x+\pi}{2\pi}\Big\rfloor
\in [-\pi,\pi).
\end{equation}
For two opinions $\alpha,\beta$ we write $\wrap(\beta-\alpha)$
for their shortest circular difference.

We consider two possible interaction graphs: the \emph{path} with edge set $\mathcal E_\emptyset=\{(i,i+1):1\le i\le N-1\}$, and the \emph{ring} with periodic boundary conditions
$\mathcal E_p=\{(i,i+1):1\le i\le N-1\}\cup\{(N,1)\}$. Hence, in this paper $\mathcal{E}\in\{\mathcal{E}_\emptyset,\mathcal{E}_p\}$. For an edge $e=(i,j)$ we define the {\it circular increment}
\begin{equation}
\label{eq:delta}
\delta_t(e):=\wrap\big(\theta_t(j)-\theta_t(i)\big).
\end{equation}
We define the winding number, which will be an integer on the ring,  in the following way:
\begin{equation}
\label{e:winding}
W_t:=\frac1{2\pi}\sum_{e\in\mathcal E_p}\delta_t(e).
\end{equation}

\subsection{\texttt{ACCA} midpoint dynamics}

We briefly summarise the \texttt{ACCA} (\emph{Asynchronous Continuous-space Cellular Automaton}) midpoint rule introduced in \cite{Brockhaus2026}. 
At each discrete time step an edge is selected uniformly at random and
the two corresponding opinions move to their circular midpoint along the shortest arc.

\begin{algorithm}[H]
\caption{\texttt{ACCA} midpoint dynamics}
\label{alg:acca}
\begin{algorithmic}[1]
\State Choose an edge $e=(i,j)$ uniformly at random from $\mathcal E$
\State $\delta \gets \wrap(\theta_t(j)-\theta_t(i))$
\State $\theta_{t+1}(i) \gets \wrap(\theta_t(i)+\delta/2)$
\State $\theta_{t+1}(j) \gets \wrap(\theta_t(j)-\delta/2)$
\State All other sites remain unchanged
\end{algorithmic}
\end{algorithm}

Thus, the selected pair is replaced by its midpoint along the shortest
arc of the circle.
This local averaging mechanism decreases disagreement and drives the
system toward consensus \cite{gantert2020}. We refer the interested reader to \cite{Brockhaus2026} for further properties of the dynamics.

\subsection{\texttt{ACCA} dynamics with \emph{bi-modal} noise}

To investigate the robustness of the dynamics, we introduce a noisy
variant in which opinions are occasionally perturbed toward two
\emph{antipodal} directions.

\begin{algorithm}[H]
\caption{\texttt{ACCA} dynamics with bi-modal noise}
\label{alg:noise}
\begin{algorithmic}[1]
\State Apply \texttt{Algorithm}~1 to configuration $\theta_t$ and obtain the configuration $\theta^\star_t$
\State Choose site $k\in V_N$ uniformly from $\{1,\dots,N\}$
\State Choose $a\in\{0,\pi\}$ with probability $1/2$
\State $\theta_{t+1}(k)\gets \wrap\left(
\theta^\star_t(k)+\varepsilon\,\wrap(a-\theta^\star_t(k))\right)$
\end{algorithmic}
\end{algorithm}

The perturbation moves the chosen opinion a fraction $\varepsilon$ of the
shortest circular distance toward the target angle $a$.
For $\varepsilon=0$ the dynamics reduces to the non-noisy \texttt{ACCA}
model, while for $\varepsilon>0$ the system experiences persistent
symmetric forcing toward the two antipodal directions $0$ and $\pi$ (see Figure~\ref{fig:acca_noise}).
\begin{figure}[tbp]
\centering
\begin{subfigure}[t]{0.48\textwidth}
\centering
\begin{tikzpicture}[scale=0.95, transform shape,
    every node/.style={font=\scriptsize}]
    \draw[thick] (0,0) circle (1);
    \fill (0,0) circle (0.02);

    \def\a{25}
    \def\b{115}
    \pgfmathsetmacro{\m}{(\a+\b)/2}

    \coordinate (A) at (\a:1);
    \coordinate (B) at (\b:1);
    \coordinate (M) at (\m:1);

    \draw[dashed] (0,0) -- (A);
    \draw[dashed] (0,0) -- (B);
    \draw[dashed] (0,0) -- (M);

    \draw[thick, blue] (\a:1) arc[start angle=\a,end angle=\b,radius=1];

    \fill (A) circle (0.04);
    \fill (B) circle (0.04);
    \fill[blue] (M) circle (0.05);

    \node[right] at ($(A)+(0.08,-0.02)$) {$\theta_t(i)$};
    \node[left]  at ($(B)+(-0.08,0)$) {$\theta_t(j)$};
    \node[above] at ($(M)+(0,0.10)$) {$\theta_{t}^{*}(i)=\theta_{t}^{*}(j)$};

    \node at (0,-1.30) {midpoint step};
\end{tikzpicture}
\caption{}
\label{fig:acca_noise_a}
\end{subfigure}
\hfill
\begin{subfigure}[t]{0.48\textwidth}
\centering
\begin{tikzpicture}[scale=0.95, transform shape,
    every node/.style={font=\scriptsize}]
    \draw[thick] (0,0) circle (1);
    \fill (0,0) circle (0.02);

    \def\th{70}
    \def\targetA{0}
    \def\targetB{180}
    \def\eps{0.35}

    \pgfmathsetmacro{\wrapA}{mod(\targetA-\th+180,360)-180}
    \pgfmathsetmacro{\wrapB}{mod(\targetB-\th+180,360)-180}

    \pgfmathsetmacro{\newA}{\th + \eps*\wrapA}
    \pgfmathsetmacro{\newB}{\th + \eps*\wrapB}

    \coordinate (X)   at (\th:1);
    \coordinate (T0)  at (\targetA:1);
    \coordinate (Tpi) at (\targetB:1);
    \coordinate (XA)  at (\newA:1);
    \coordinate (XB)  at (\newB:1);

    \draw[dashed] (0,0) -- (X);
    \draw[dashed] (0,0) -- (T0);
    \draw[dashed] (0,0) -- (Tpi);

    \fill[gray] (T0) circle (0.035);
    \fill[gray] (Tpi) circle (0.035);

    \fill (X) circle (0.045);

    \draw[->, thick, blue]
        (\th:1) arc[start angle=\th,end angle=\newA,radius=1];
    \draw[->, thick, purple]
        (\th:1) arc[start angle=\th,end angle=\newB,radius=1];

    \fill[blue] (XA) circle (0.05);
    \fill[purple] (XB) circle (0.05);

    \node[above right] at ($(X)+(0.03,0.04)$) {$\theta_t^\star(k)$};
    \node[right] at ($(T0)+(0.10,0)$) {$0$};
    \node[left]  at ($(Tpi)+(-0.10,0)$) {$\pi$};

    \node[blue, below right] at ($(XA)+(0.2,-0.02)$) {towards $0$};
    \node[purple, above left] at ($(XB)+(-0.02,0.02)$) {towards $\pi$};

    \node at (0,-1.30) {bi-modal noise step};
\end{tikzpicture}
\caption{}
\label{fig:acca_noise_b}
\end{subfigure}

\caption{One update of the \texttt{ACCA} dynamics under \texttt{Algorithm~2}.
(a) A randomly chosen edge $(i,j)$ is updated by the circular midpoint rule, producing an intermediate state $\theta_{t}^{*}$.
(b) Then a site $k$ is selected uniformly and its opinion is moved a fraction $\varepsilon$ toward one of the two antipodal targets $a\in\{0,\pi\}$, each chosen with probability $1/2$. In the schematic, the displacement is shown with $\varepsilon=0.35$ for visibility.}
\label{fig:acca_noise}
\end{figure}
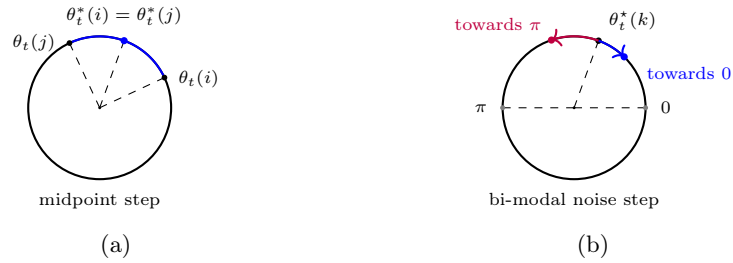
\subsection{\texttt{ACCA} dynamics with parallel updates and bi-modal noise}

We now introduce a generalization of the \texttt{ACCA} dynamics in which both the local averaging mechanism and the stochastic perturbation act simultaneously on multiple sites. This variant interpolates between the original sequential dynamics and a more parallel, cellular-automaton-like regime, allowing one to investigate how collective updates affect the formation and stability of coherent structures.

At each time step, a collection of disjoint edges is selected and updated in parallel according to the midpoint rule, producing an intermediate configuration. A stochastic perturbation is then applied in parallel to a randomly chosen subset of sites. The sizes of these subsets are fixed parameters of the dynamics, controlling the degree of parallelism in the deterministic and stochastic components.

\begin{algorithm}[H]
\caption{\texttt{ACCA} dynamics with parallel updates and bi-modal noise}
\label{alg:burst-acca}
\begin{algorithmic}[1]
\State Fix integers $k_{\mathrm{mid}} \geq 1$ and $k_{\mathrm{noise}} \geq 0$
\State Sample a set $\mathcal F \subseteq \mathcal E$ with $|\mathcal{F}|=k_{\mathrm{mid}}$ of disjoint edges uniformly at random
\State Set $\theta_t^\star \gets \theta_t$
\For{each edge $(i,j) \in \mathcal F$}
    \State apply \texttt{Algorithm}~1, obtaining $\theta^*_t(i), \theta^*_t(j)$
\EndFor
\State Sample a subset $B \subseteq V_N$ with $|B|=k_{\mathrm{noise}}$ uniformly at random
\State Set $\theta_{t+1} \gets \theta_t^\star$
\For{each $i \in B$}
    \State apply \texttt{Algorithm}~2, obtaining $\theta_{t+1}(i)$
\EndFor
\end{algorithmic}
\end{algorithm}

In this formulation, the midpoint interaction is applied simultaneously to a collection of $k_{\mathrm{mid}}$ disjoint edges. The disjointness condition ensures that each site is involved in at most one update, so that the parallel step is well defined. All midpoint moves are computed from the same configuration $\theta_t$, yielding an intermediate state $\theta_t^\star$.

Starting from this intermediate configuration, a stochastic perturbation is applied in parallel to a subset $B$ of $k_{\mathrm{noise}}$ sites, chosen uniformly without replacement. Each selected site is updated independently: a target angle $a \in \{0,\pi\}$ is drawn with probability $1/2$, and the opinion is moved a fraction $\varepsilon$ of the shortest circular distance towards this target. Sites not in $B$ remain unchanged.

For $k_{\mathrm{mid}}=1$ and $k_{\mathrm{noise}}=1$, the dynamics reduces to the standard sequential \texttt{ACCA} model with bi-modal noise defined in \texttt{Algorithm}~2. Increasing $k_{\mathrm{mid}}$ enhances the degree of parallel local averaging, while increasing $k_{\mathrm{noise}}$ produces more spatially extended and independent stochastic perturbations, which allows to interpolate between asynchronous and partially synchronous regimes.
\subsection{Properties of \texttt{ACCA} midpoint dynamics}
The qualitative behavior of the \texttt{ACCA} midpoint dynamics is well understood (see \cite{Brockhaus2026}). 
On finite connected graphs with open boundary conditions the system converges almost surely to consensus, as proved for related midpoint models in \cite{gantert2020} and for the \texttt{ACCA} dynamics in our previous work \cite{Brockhaus2026}.

Under periodic boundary conditions the situation is different. 
Because of the circular geometry, configurations can organize into metastable (trapped) twisted profiles characterized by a non–zero winding number $W_t$. 
In such states the opinions form an approximately linear phase profile along the ring. 
Although these configurations are not stationary, the midpoint averaging dynamics keeps the system close to the twisted manifold for very long times, producing a form of dynamical trapping.

When stochastic perturbations are introduced ($\varepsilon>0$), the dynamics changes qualitatively. 
The bi-modal forcing toward the antipodal angles $0$ and $\pi$ gradually destabilizes the twisted profiles and may induce collective switching between preferred orientations. 
Numerical experiments reveal a characteristic \emph{spin--flop} regime in which the macroscopic state alternates between two coherent orientations.
\section{Order parameters}
\label{sec:order}
The coexistence of winding trapping and antipodal switching requires two distinct macroscopic diagnostics: one sensitive to topological structure and one sensitive to global orientation.

In order to detect the onset of consensus, we first recall known  \emph{order parameters}.
Analogously to what is typically done for the Kuramoto model, as for example in \cite{Strogatz2000}, we can investigate the order parameters $R(t)$ and $\psi(t)$ defined as follows:
\begin{definition}[Order parameters $R$ and $\psi$]
 Let $\theta_t\in \Omega_N$ denote a configuration at time $t$. We then define the order parameters $R(t)$ and $\psi(t)$ by:

 \begin{equation}
 \label{e:order1}
R(t) e^{i \psi(t)}:=\frac{1}{N} \sum_{j=1}^N e^{i \theta_t(j)}.
\end{equation}

\end{definition}
The order parameter $R$ quantifies the degree of alignment among opinions, while $\psi$ represents the average opinion. Each opinion is represented as a direction and is represented as a point on the unit circle. By summing these directions as vectors, we obtain a resultant vector pointing in the average direction $\psi$, whose magnitude $R$ indicates the strength of alignment. If $R=0$, there is no alignment, i.e. the opinions are uniformly distributed around the circle. If $R=1$ there is perfect consensus, which means that all opinions are identical. Thus, the closer $R$ is to 1, the more synchronized the opinions are. In the noisy extension with symmetric bi-modal perturbations toward the two
antipodal angles $0$ and $\pi$, the dynamics may display a
\emph{spin--flop regime}. We use the term \emph{spin--flop regime} to denote a dynamics in which
the system spends long periods of time near one of two metastable
collective orientations before abruptly switching to the other. 
In this regime the system concentrates close to the two states
\[
\theta_i \approx \frac{\pi}{2}
\qquad\text{or}\qquad
\theta_i \approx -\frac{\pi}{2},
\qquad i=1,\dots,N.
\]
 In the observable
$Y(t)$, the phenomenon appears as long positive and negative plateaus
separated by relatively rapid transitions.
Such configurations are not well detected by the classical Kuramoto
order parameter $R(t)$ alone, since $R(t)$ only measures the degree of
alignment but not the orientation of the collective state.
To distinguish the two branches it is convenient to consider the
macroscopic projection onto the $y$ axis,
\begin{equation}
\label{eq:flipflop}
Y(t):=\frac1N\sum_{j=1}^N \sin(\theta_t(j)).
\end{equation}
Configurations concentrated near $\theta=\frac{\pi}{2}$ yield
$Y(t)\approx 1$, while configurations concentrated near
$\theta=-\frac{\pi}{2}$ give $Y(t)\approx -1$.
Therefore the sign of $Y(t)$ distinguishes the two \emph{spin-flop}
orientations, and transitions of $Y(t)$ between positive and negative
plateaus provide a clear macroscopic signature of the spin--flop
dynamics.

We recall the definition of \emph{winding configurations} given in \cite{Brockhaus2026} and graphically shown in Figure~\ref{fig:100_periodic_eps000} panael (a).
\begin{definition}[Winding configuration]
Let $\theta\in[-\pi,\pi)^N$ be a configuration on $\mathcal{E}_p$.
Its winding number is defined by
\[
W(\theta)
:=
\frac1{2\pi}
\sum_{(i,j)\in \mathcal{E}_p}
\operatorname{wrap}_{\pi}\!\bigl(\theta(j)-\theta(i)\bigr).
\]
We say that $\theta$ is a winding configuration with winding number
$W\in\mathbb Z$ if $W(\theta)=W$.
\end{definition}
A winding configuration represents a coherent twisted profile around the
ring: after traversing the graph once, the lifted phase accumulates a
total increment of $2\pi W$. In our previous work \cite{Brockhaus2026},
such configurations were shown to organize the transient dynamics into
topological sectors labelled by the winding number, leading to long-lived
\emph{metastable} states and topological trapping before eventual convergence
to consensus.
In order to detect winding configurations we introduce an order parameter
based on a nonparametric correlation coefficient on the torus,
analogous to Kendall's $\tau$ (see \cite{Zhang2019}).

This observable measures the statistical alignment between two angular
variables while taking into account the periodic geometry.

Let $P_1=(\theta_1,\phi_1)$ and $P_2=(\theta_2,\phi_2)$ be two points on the
torus. Intuitively, the pair is said to be \emph{concordant} if the circular
order from $\theta_1$ to $\theta_2$ and from $\phi_1$ to $\phi_2$ has the
same orientation, and \emph{discordant} otherwise.

\begin{definition}[Weighted toroidal Kendall coefficient]
Let $(\Theta_1,\Phi_1)$ and $(\Theta_2,\Phi_2)$ be i.i.d.\ copies of
$(\Theta,\Phi)$. The \emph{weighted toroidal Kendall coefficient} is defined as:
\begin{equation}
\label{eq:tau1-def}
\tau_1
:=
\frac{\mathbb{E}\!\left[\wrap(\Theta_2-\Theta_1)
\mathrm{sign}\!\big(\wrap(\Phi_2-\Phi_1)\big)\right]}
{\mathbb{E}\!\left[|\wrap(\Theta_2-\Theta_1)|\right]},
\end{equation}
where it is assumed that the denominator is non-zero.
\end{definition}
The normalization ensures that $\tau_1\in[-1,1]$.
The coefficient $\tau_1$ measures the correlation between the circular
ordering of $\Theta$ and that of $\Phi$, while preserving the magnitude
of the wrapped difference.

We now compute $\tau_1$ for an idealized probabilistic model of a winding configuration.
Let
\[
\Theta_1,\Theta_2
\overset{\mathrm{i.i.d.}}{\sim}
\mathrm{Unif}(-\pi,\pi],
\]
and define

\begin{equation}
\label{eq:twist-ideal}
\Phi_i
:=
\alpha+W\Theta_i
\pmod{2\pi},
\qquad i\in\{1,2\},
\end{equation}
where $W\in\mathbb Z$ and $\alpha\in[-\pi,\pi)$.
We have the following theorem:
\begin{theorem}
\label{th:main}
Under the model \eqref{eq:twist-ideal},
for $W\neq0$ one has
\begin{equation}
\label{eq:tau1-final}
\tau_1=\frac{(-1)^{W+1}}{W}.
\end{equation}
In particular
$|\tau_1|=\frac{1}{|W|}$,
for  $W\neq0.$
\end{theorem}
\begin{proof}
Let
\[
\tilde\Delta :=  \wrap(\Theta_2-\Theta_1).
\]

Since $\Theta_1,\Theta_2$ are independent and uniformly distributed on
$(-\pi,\pi]$, the wrapped difference $\tilde\Delta$ is uniformly distributed
on $(-\pi,\pi]$.
Indeed, if $X=\Theta_2-\Theta_1$, the density of $X$ is triangular on
$(-2\pi,2\pi)$:
\[
f_X(x)=\frac{2\pi-|x|}{4\pi^2}.
\]
Wrapping $X$ into $(-\pi,\pi]$ folds the density onto this interval,
and a direct computation shows that the resulting density is constant
$1/(2\pi)$ (see proof Lemma~4.5 in \cite{Brockhaus2026}). Under the twisted assumption \eqref{eq:twist-ideal} we have $
\wrap(\Phi_2-\Phi_1)
=\wrap(W\tilde\Delta)$.
Since $\tilde\Delta$ has a continuous distribution, the event
$\tilde\Delta\in\pi\mathbb Z$ has probability zero. 
Moreover, $\wrap(x)\in[-\pi,\pi)$, one has 
$\mathrm{sign}(\wrap(x))=\mathrm{sign}(\sin x)$ for all $x\notin\pi\mathbb{Z}$.
Hence,
\[
\mathrm{sign}\!\big(\wrap(\Phi_2-\Phi_1)\big)
=
\mathrm{sign}\big(\sin(W\tilde\Delta)\big), \,\,\textnormal{a.s.}
\]
Therefore
\begin{equation}
\label{eq:tau1-int}
\mathbb{E}\!\left[
\tilde\Delta\,\mathrm{sign}(\wrap(\Phi_2-\Phi_1))
\right]
=
\frac{1}{2\pi}
\int_{-\pi}^{\pi}
x\,\mathrm{sign}(\sin(Wx))\,dx .
\end{equation}

Since the integrand is even, this becomes
\[
\frac{1}{\pi}\int_0^\pi x\,\mathrm{sign}(\sin(Wx))\,dx .
\]

On $(0,\pi)$ the function $\sin(Wx)$ changes sign at
$x_k=k\pi/W$ for $k=0,\dots,W$.
Hence
\[
\int_0^\pi x\,\mathrm{sign}(\sin(Wx))\,dx
=
\sum_{k=0}^{W-1}
(-1)^k
\int_{k\pi/W}^{(k+1)\pi/W}x\,dx .
\]
A direct computation gives
\[
\int_{k\pi/W}^{(k+1)\pi/W}x\,dx
=
\frac{(2k+1)\pi^2}{2W^2}.
\]
Thus
\[
\int_0^\pi x\,\mathrm{sign}(\sin(Wx))\,dx
=
\frac{\pi^2}{2W^2}
\sum_{k=0}^{W-1}(-1)^k(2k+1).
\]
The alternating sum evaluates to
\[
\sum_{k=0}^{W-1}(-1)^k(2k+1)
=
(-1)^{W+1}W .
\]
Hence
\[
\mathbb{E}\!\left[
\tilde\Delta\,\mathrm{sign}(\wrap(\Phi_2-\Phi_1))
\right]
=
\frac{(-1)^{W+1}\pi}{2W}.
\]
Finally,
\[
\mathbb{E}[|\tilde\Delta|]
=
\frac{1}{2\pi}\int_{-\pi}^{\pi}|x|dx
=
\frac{\pi}{2},
\]
so that: $\tau_1=\frac{(-1)^{W+1}}{W}$.\hspace{8.3cm} $\blacksquare$
\end{proof}
Theorem~\ref{th:main} provides the theoretical justification for using the weighted toroidal Kendall coefficient as a macroscopic diagnostic of topological organization in the \texttt{ACCA} dynamics. In particular, it establishes an explicit quantitative relation between the winding number of an ideal twisted configuration and the observable $\tau_1$, showing that distinct winding sectors produce characteristic inverse-signature values. This result is important because it allows one to detect, distinguish, and monitor winding trapping directly from simulation data without reconstructing the full microscopic topological structure. Consequently, $\tau_1$ serves as a practical order parameter linking the rigorous topological theory of the deterministic model to the numerical exploration of noisy and partially parallel extensions developed in the present work.
$\tau_1$ detects indeed twisted configurations through a
characteristic inverse dependence on the winding number. 
Despite the fact that the classical Kuramoto order parameter $R(t)$
is well suited to detect convergence to global consensus, it does not distinguish between a flat profile and a coherent
twisted configuration with nonzero winding number.
In a trapping winding state the configuration is highly ordered
along the ring but not phase–aligned, and therefore $R(t)$ may remain
moderate even though the system is strongly structured.
\section{Results}
\label{sec:results}
All the experiments in this section are started from a random configuration $\theta_0\in\Omega_N$ with $\theta_0(i)\overset{i.i.d.}{\sim} \textnormal{Unif}[-\pi,\pi)$, for $i\in V_N$, where $N=100$.

Figure~\ref{fig:100_periodic_eps000} illustrates the midpoint \texttt{ACCA} dynamics of \texttt{Algorithm~1} under periodic boundary conditions (i.e., $\mathcal{E}=\mathcal{E}_p$) for a configuration with winding number $W=2$. 
In particular, in the panel (b) the space--time heatmap (top left) shows that the system rapidly organizes into a twisted configuration characterized by an approximately linear phase profile along the ring. 
The winding structure is clearly detected by the circular Kendall coefficient $\tau_1(t)$ (bottom right), which quickly converges to a constant value consistent with the theoretical prediction $|\tau_1|=1/|W|=1/2$. 
At the same time the classical Kuramoto order parameter $R(t)$ (top right) remains small, indicating that the configuration is highly ordered along the ring but not phase--aligned, as expected for a twisted state. 
Finally, the macroscopic orientation observable $Y(t)$ (bottom left) fluctuates around zero, which reflects the absence of a preferred global orientation in the deterministic twisted configuration. 
Together, these observables clearly distinguish winding trapping from global consensus and provide complementary diagnostics of the system's macroscopic organization.

\begin{figure}[t]
\centering

\makebox[\textwidth][c]{%
\begin{subfigure}[t]{0.515\textwidth}
    \centering
    \includegraphics[
        width=1.06\linewidth,
        trim=0.2cm 0.2cm 0.2cm 0.2cm,
        clip
    ]{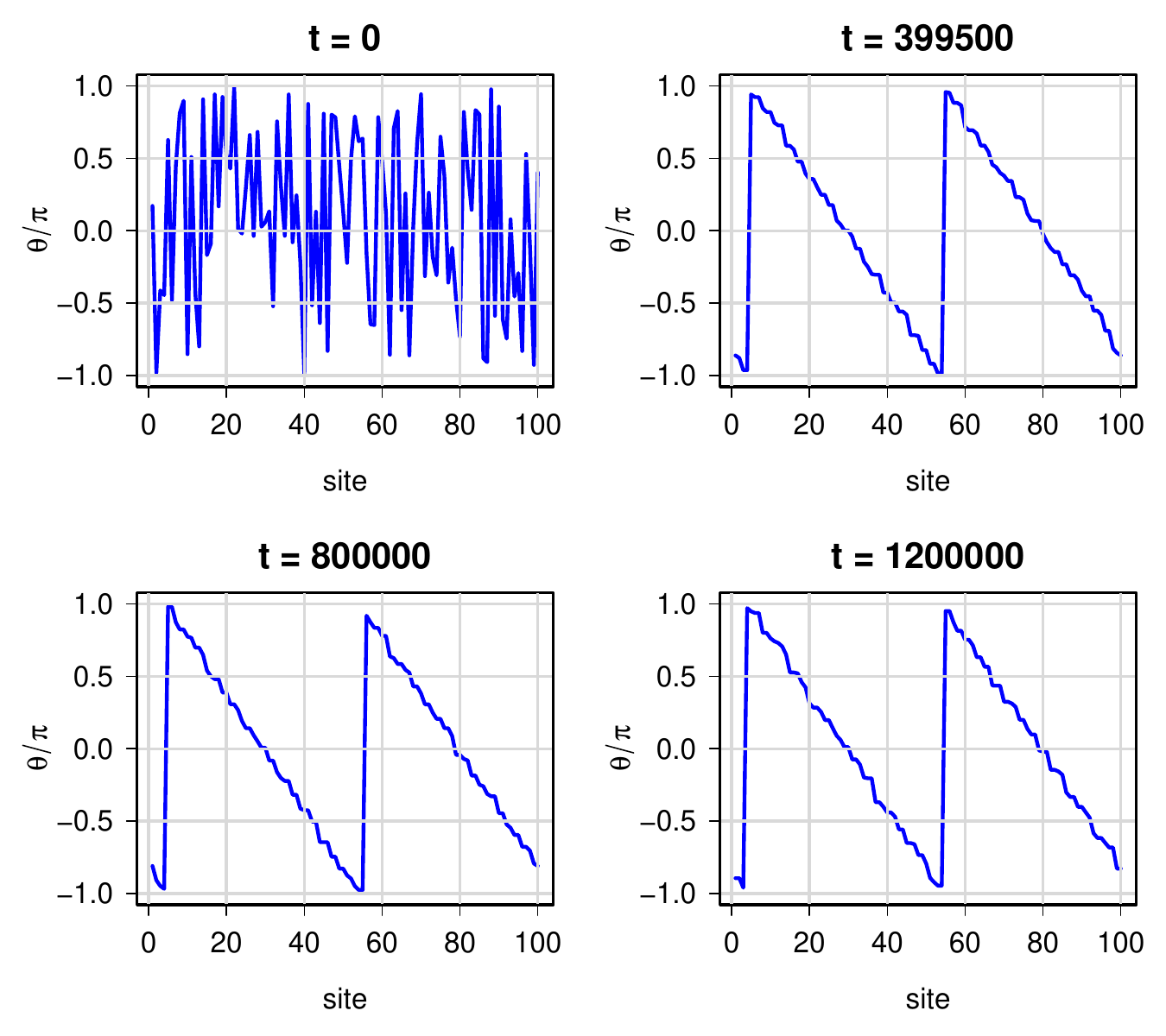}
    \caption*{(a)}
    \label{fig:periodic_eps000_a}
\end{subfigure}
\hspace{0.01\textwidth}
\begin{subfigure}[t]{0.515\textwidth}
    \centering
    \includegraphics[
        width=1.06\linewidth,
        trim=0.2cm 0.2cm 0.2cm 0.2cm,
        clip
    ]{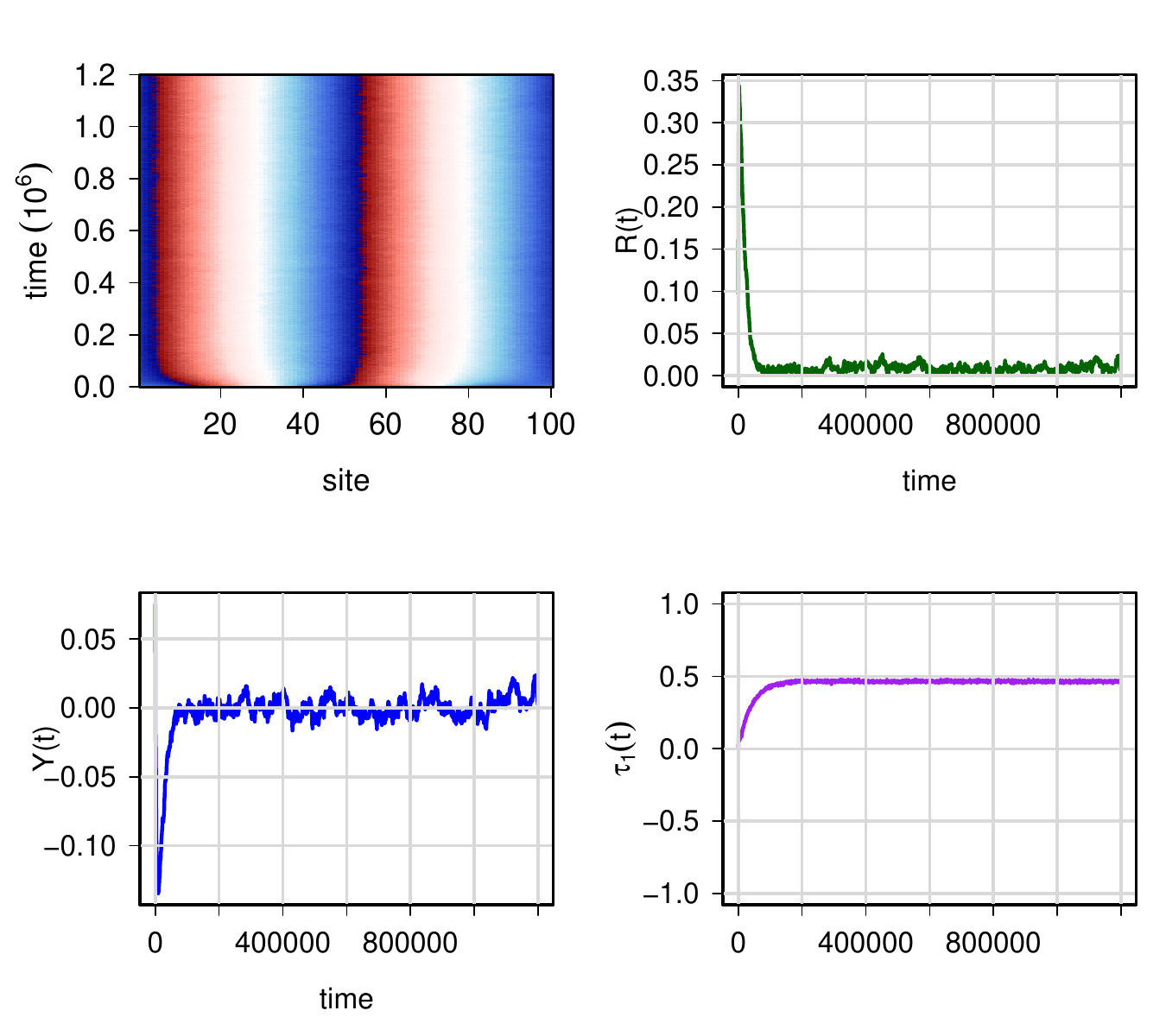}
    \caption*{(b)}
    \label{fig:periodic_eps000_b}
\end{subfigure}
}

\vspace{-0.4em}

\caption{
ACCA dynamics with $N=100$ and $\varepsilon=0$ under periodic boundary conditions.
(a) Configuration snapshots at four different times.
(b) Heatmap of the time evolution together with the order parameters
$R(t)$, $Y(t)$ and $\tau_1(t)$.
}
\label{fig:100_periodic_eps000}

\vspace{-0.6em}

\end{figure}
Figure~\ref{fig:100_periodic_eps} illustrates the effect of bi-modal perturbations (i.e. \texttt{Algorithm~2}) under periodic boundary conditions.
In contrast with the deterministic case, where the configuration remains trapped near a twisted ramp associated with a fixed winding number, the presence of noise progressively destabilizes this structure.
For weak perturbations (e.g.,\ $\varepsilon=0.002$), the ramp remains clearly visible in the space--time heatmap and the configuration stays close to a coherent winding profile for long times.
As the noise intensity increases (e.g.,\ $\varepsilon=0.02$), however, the ramp is gradually destroyed and the system no longer maintains a stable linear profile along the ring. The resulting dynamics is not characterized by well-formed plateaus around the two spin--flop states.
Rather, the observable $Y(t)$ displays large collective fluctuations, indicating repeated global reorientations of the configuration.
These fluctuations occasionally bring the system close to the two preferred orientations, but without producing the clear  plateaus observed under open boundary conditions.
At the same time, the classical Kuramoto order parameter $R(t)$ remains relatively small, showing that the system does not approach global consensus, while the toroidal Kendall coefficient $\tau_1(t)$ progressively departs from the deterministic value associated with the winding sector.
This indicates that bi-modal perturbations gradually erode the topological trapping induced by the \texttt{ACCA} dynamics: the winding structure becomes less stable, and the evolution is increasingly dominated by large-scale fluctuations rather than by confinement near a coherent twisted ramp.


\begin{figure}[t]
\centering

\makebox[\textwidth][c]{%
\begin{subfigure}[t]{0.515\textwidth}
    \centering
    \includegraphics[
        width=1.06\linewidth,
        trim=0.2cm 0.2cm 0.2cm 0.2cm,
        clip
    ]{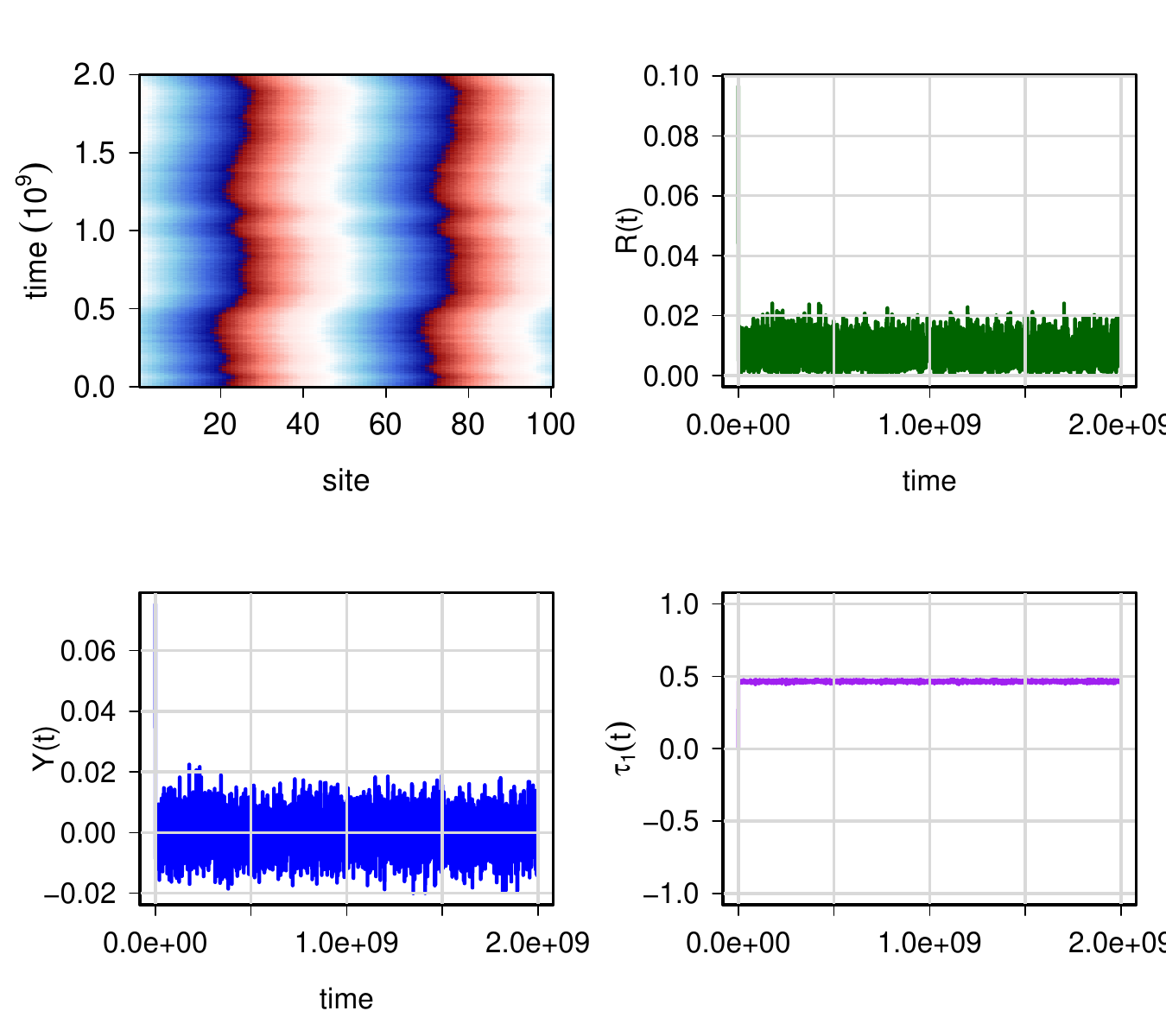}
    \caption*{(a) $\varepsilon=0.002$}
    \label{fig:periodic_eps0002_a}
\end{subfigure}
\hspace{0.01\textwidth}
\begin{subfigure}[t]{0.515\textwidth}
    \centering
    \includegraphics[
        width=1.06\linewidth,
        trim=0.2cm 0.2cm 0.2cm 0.2cm,
        clip
    ]{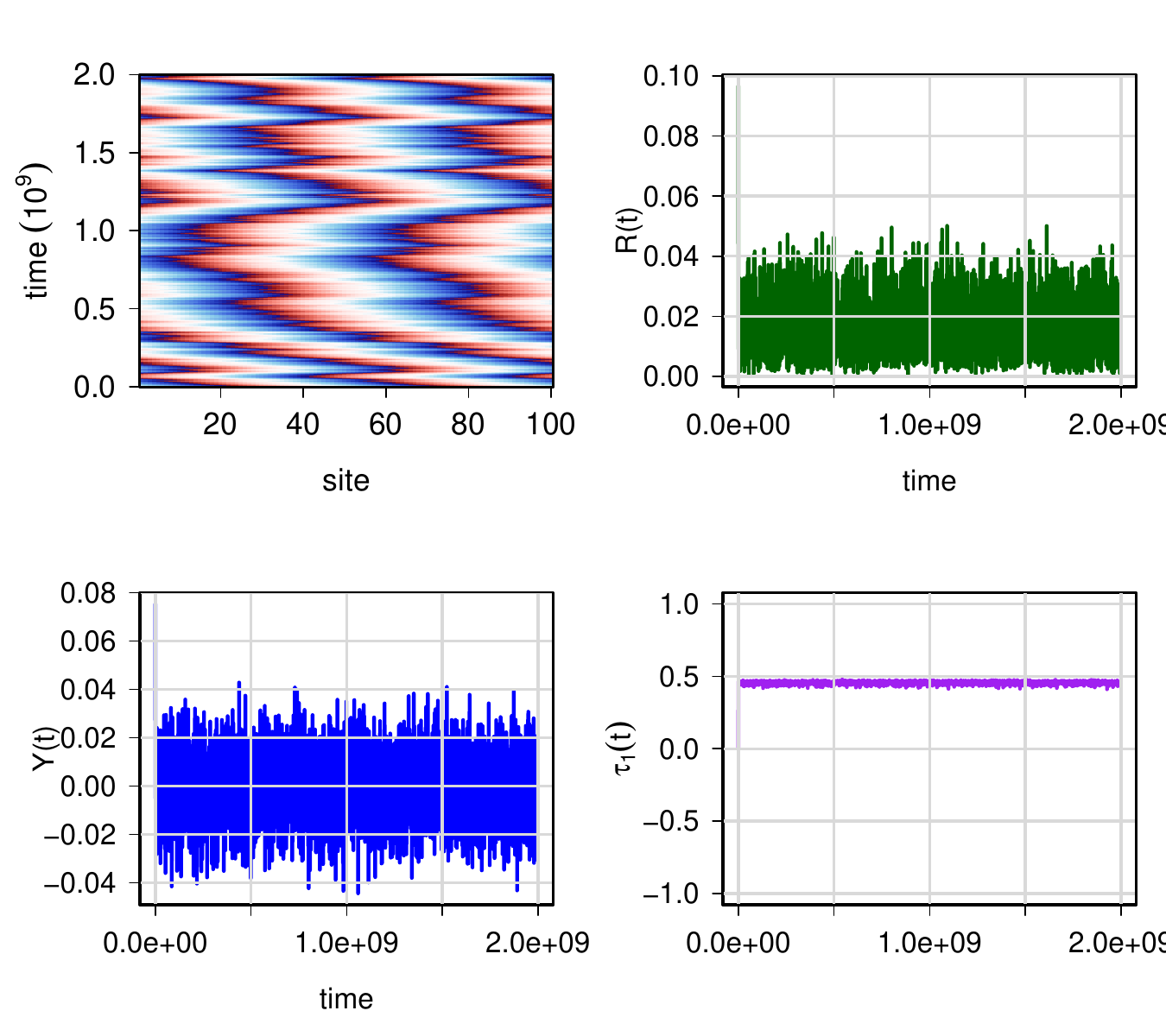}
    \caption*{(b) $\varepsilon=0.02$}
    \label{fig:periodic_eps002_b}
\end{subfigure}
}

\vspace{-0.4em}

\caption{
Heatmap of the time evolution together with the order parameters
$R(t)$, $Y(t)$ and $\tau_1(t)$ for the \texttt{ACCA} dynamics
with $N=100$ under periodic boundary conditions and non-zero noise.
Panel (a) corresponds to $\varepsilon=0.002$, while panel (b)
corresponds to $\varepsilon=0.02$.
}
\label{fig:100_periodic_eps}

\vspace{-0.6em}

\end{figure}
Figures~\ref{fig:100_open_eps000}--\ref{fig:100_open_eps0002} illustrate the dynamics under open boundary conditions (i.e., $\mathcal{E}=\mathcal{E}_\emptyset$).
In the deterministic case ($\varepsilon=0$), the system converges to consensus, as established in \cite{gantert2020,Brockhaus2026}. The introduction of bi-modal perturbations leads to a qualitatively different regime.
Even for very small noise intensity, the configuration rapidly organizes into states concentrated near the two antipodal orientations $\pm\pi/2$.
This reflects the breaking of rotational symmetry induced by the bi-modal noise. The observable $Y(t)$ clearly captures this behavior:
it alternates between positive and negative plateaus, corresponding to configurations concentrated near $\theta\approx\pi/2$ and $\theta\approx-\pi/2$.
These plateaus indicate the presence of metastable macroscopic states. In contrast with the periodic case, no topological constraint is present, and the dynamics does not involve winding trapping.
Instead, the behavior is dominated by the competition between midpoint averaging, which promotes local coherence, and bi-modal noise, which selects the two preferred orientations.
As a result, the system repeatedly reorganizes around $\pm\pi/2$, producing a clear and robust spin--flop dynamics. Importantly, these states are not stationary: the configuration continues to evolve within each plateau, reflecting the persistent action of both averaging and noise.

\begin{figure}[t]
\centering

\makebox[\textwidth][c]{%
\begin{subfigure}[t]{0.515\textwidth}
    \centering
    \includegraphics[
        width=1.06\linewidth,
        trim=0.2cm 0.2cm 0.2cm 0.2cm,
        clip
    ]{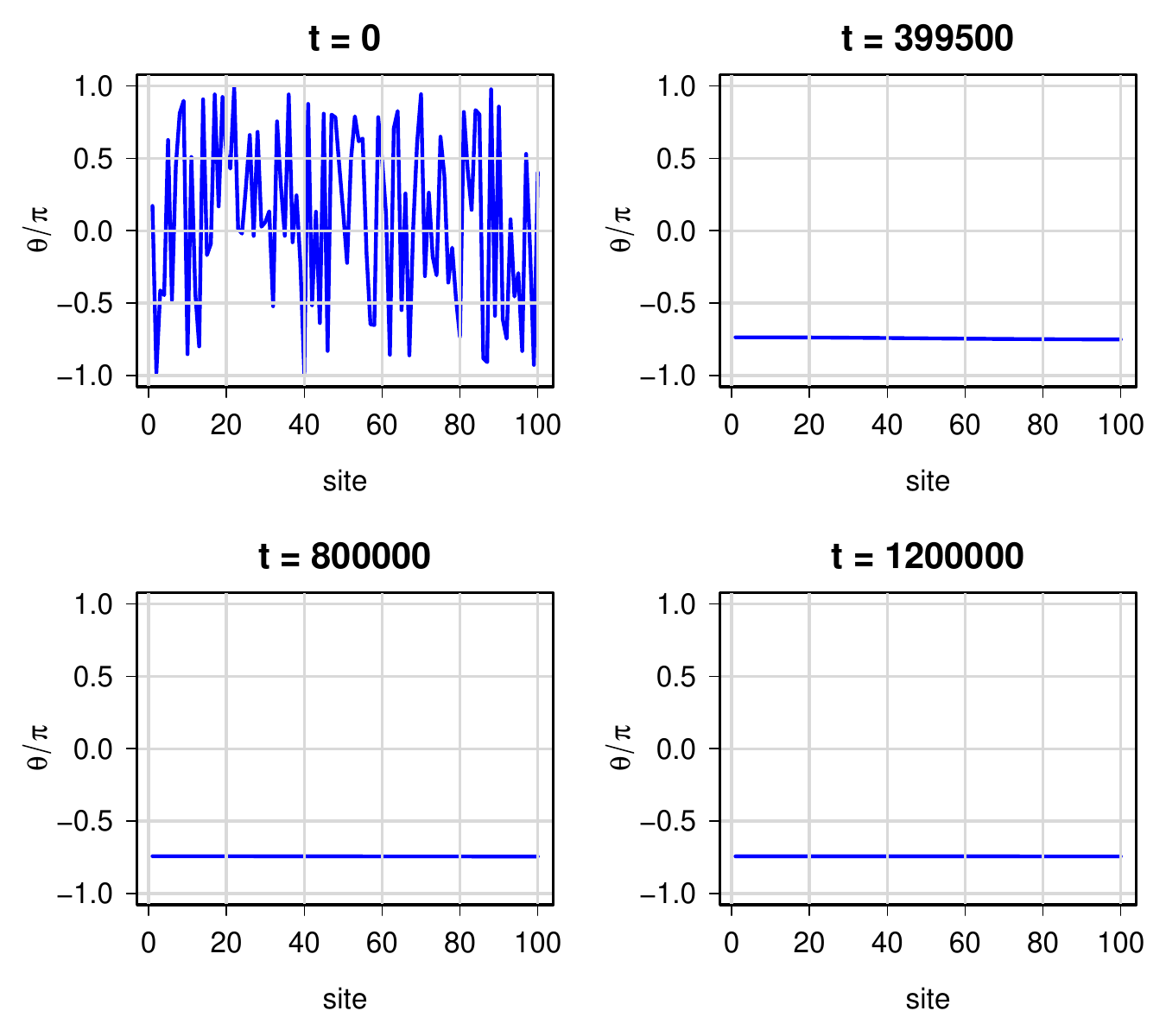}
    \caption*{(a)}
    \label{fig:open_eps000_a}
\end{subfigure}
\hspace{0.01\textwidth}
\begin{subfigure}[t]{0.515\textwidth}
    \centering
    \includegraphics[
        width=1.06\linewidth,
        trim=0.2cm 0.2cm 0.2cm 0.2cm,
        clip
    ]{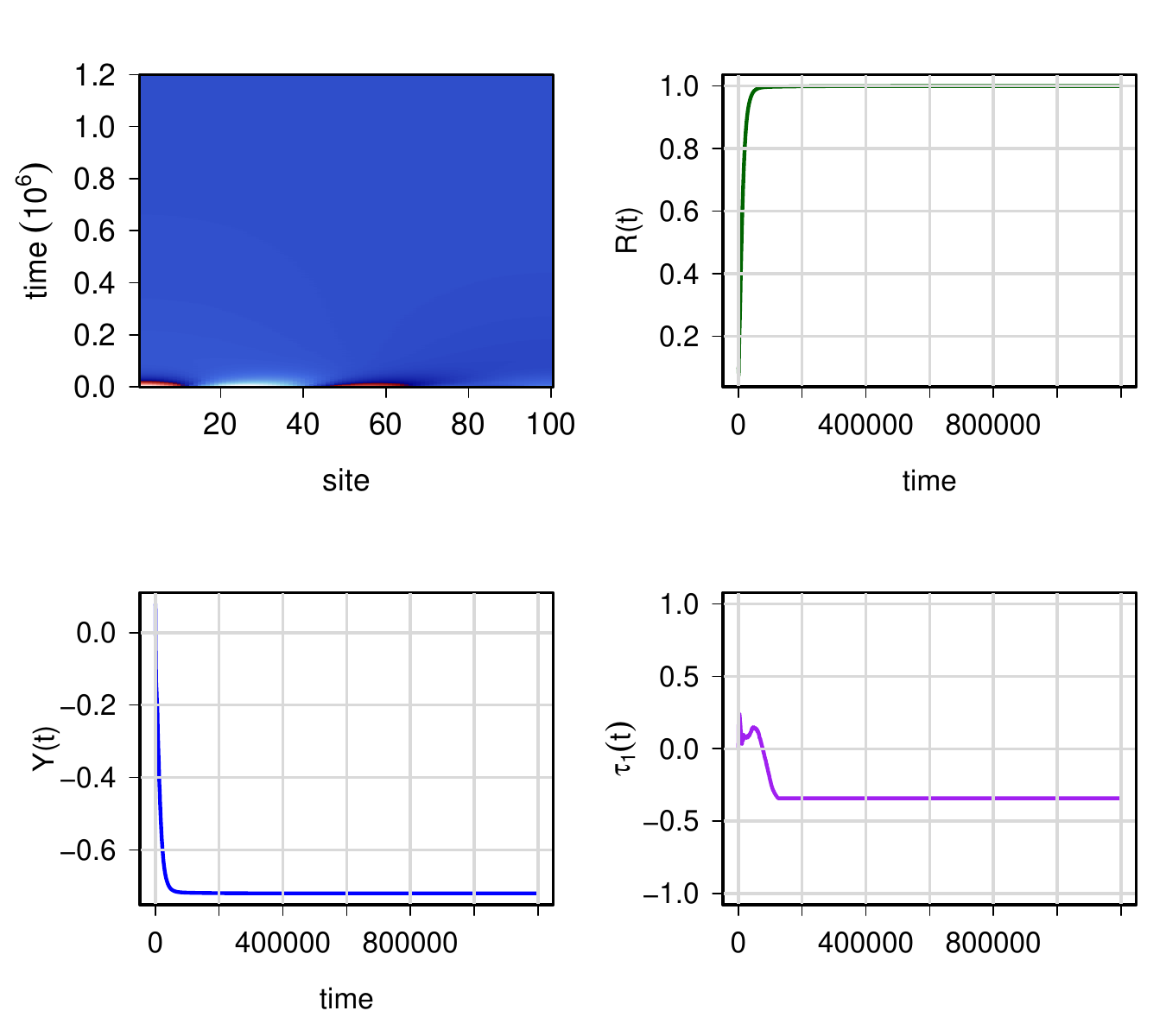}
    \caption*{(b)}
    \label{fig:open_eps000_b}
\end{subfigure}
}

\vspace{-0.4em}

\caption{
ACCA dynamics with $N=100$ and $\varepsilon=0$ under open boundary conditions.
(a) Configuration snapshots at four different times.
(b) Heatmap of the time evolution together with the order parameters
$R(t)$, $Y(t)$ and $\tau_1(t)$.
}
\label{fig:100_open_eps000}

\vspace{-0.6em}

\end{figure}


\begin{figure}[t]
\centering

\makebox[\textwidth][c]{%
\begin{subfigure}[t]{0.515\textwidth}
    \centering
    \includegraphics[
        width=1.06\linewidth,
        trim=0.2cm 0.2cm 0.2cm 0.2cm,
        clip
    ]{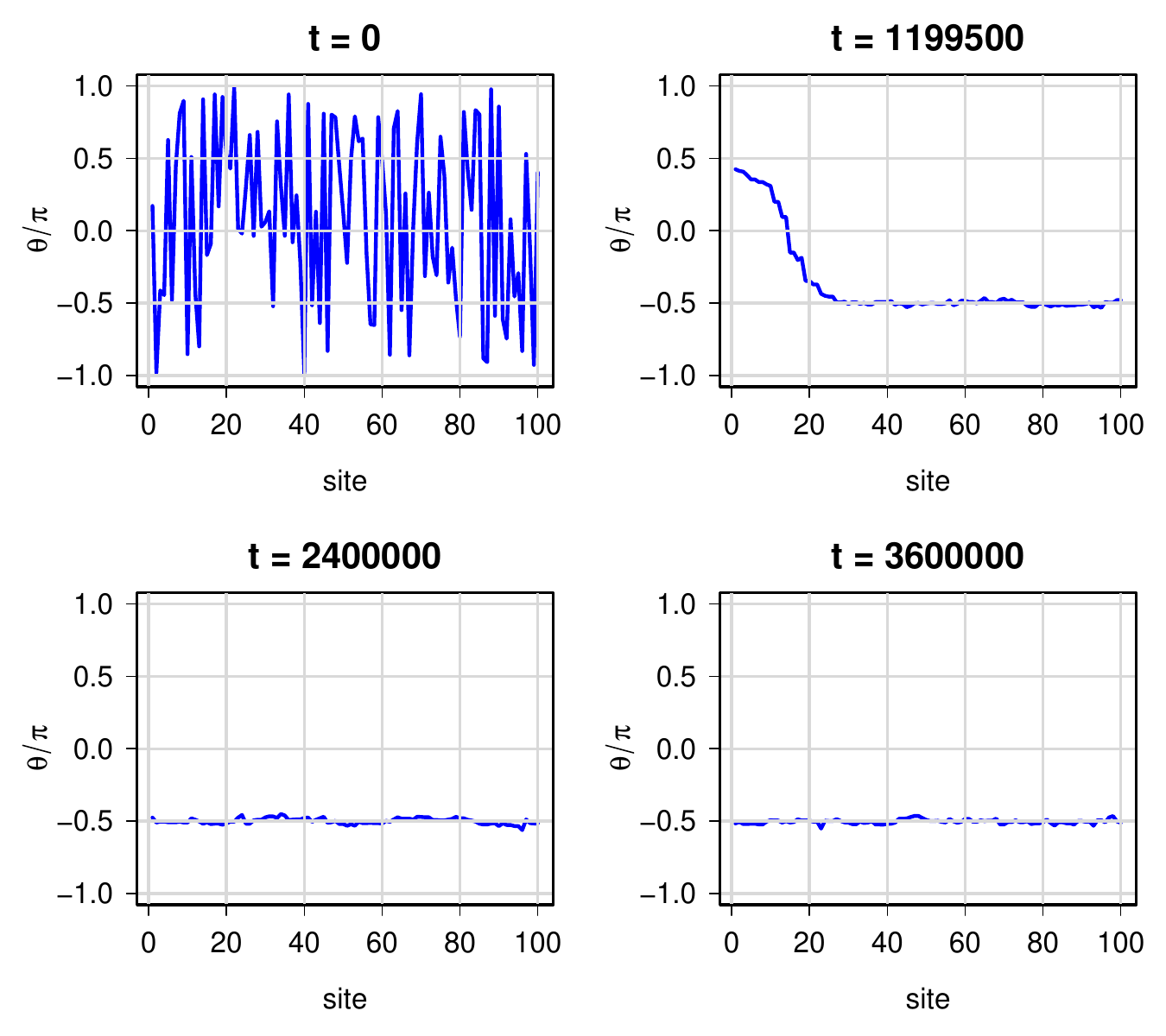}
    \caption*{(a)}
    \label{fig:open_eps002_a}
\end{subfigure}
\hspace{0.01\textwidth}
\begin{subfigure}[t]{0.515\textwidth}
    \centering
    \includegraphics[
        width=1.06\linewidth,
        trim=0.2cm 0.2cm 0.2cm 0.2cm,
        clip
    ]{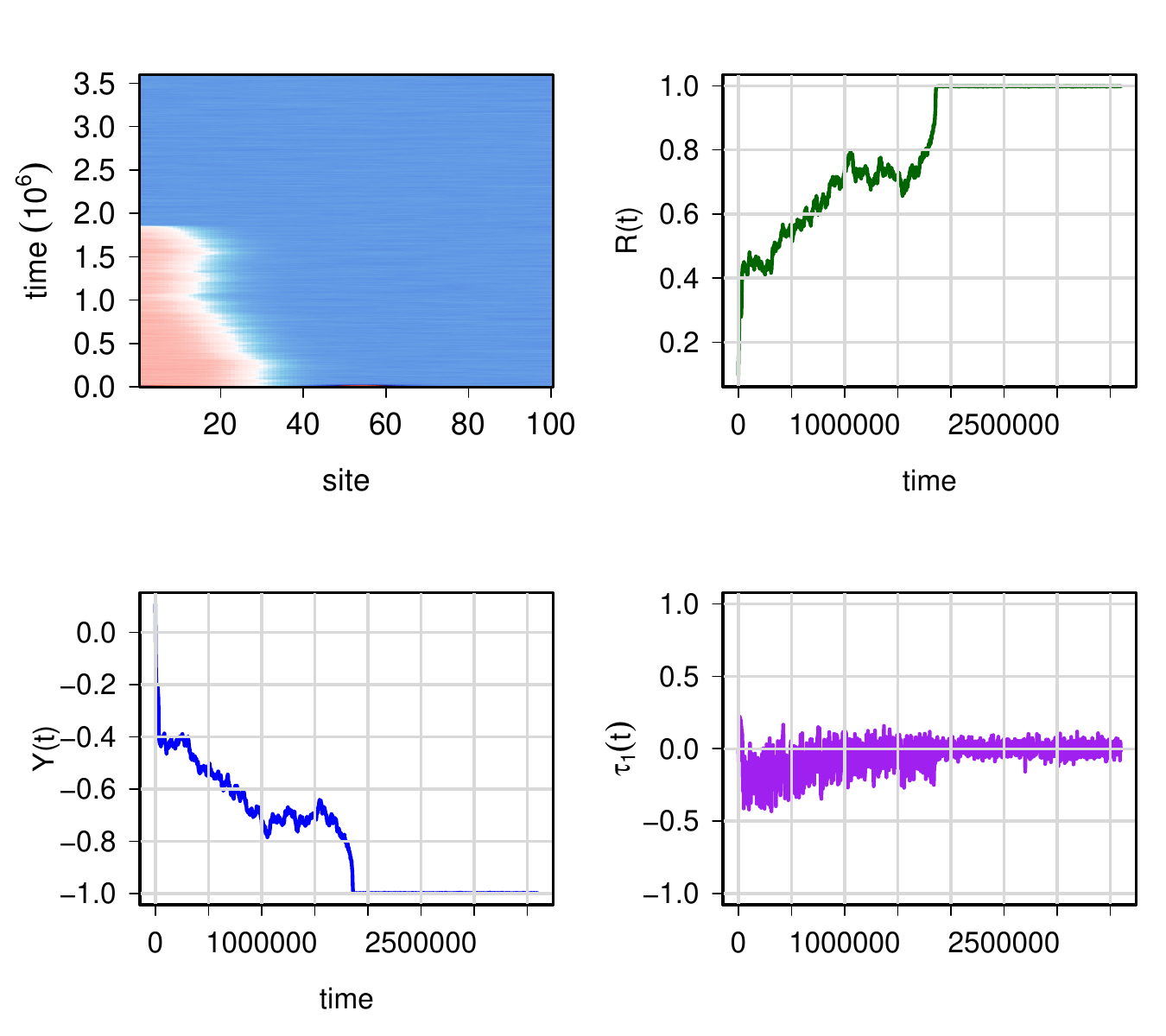}
    \caption*{(b)}
    \label{fig:open_eps002_b}
\end{subfigure}
}

\vspace{-0.4em}

\caption{
ACCA dynamics with $N=100$ and $\varepsilon=0.02$ under open boundary conditions.
(a) Configuration snapshots at four different times.
(b) Heatmap of the time evolution together with the order parameters
$R(t)$, $Y(t)$ and $\tau_1(t)$.
}
\label{fig:100_open_eps002}

\vspace{-0.6em}

\end{figure}


\begin{figure}[t]
\centering

\makebox[\textwidth][c]{%
\begin{subfigure}[t]{0.515\textwidth}
    \centering
    \includegraphics[
        width=1.06\linewidth,
        trim=0.2cm 0.2cm 0.2cm 0.2cm,
        clip
    ]{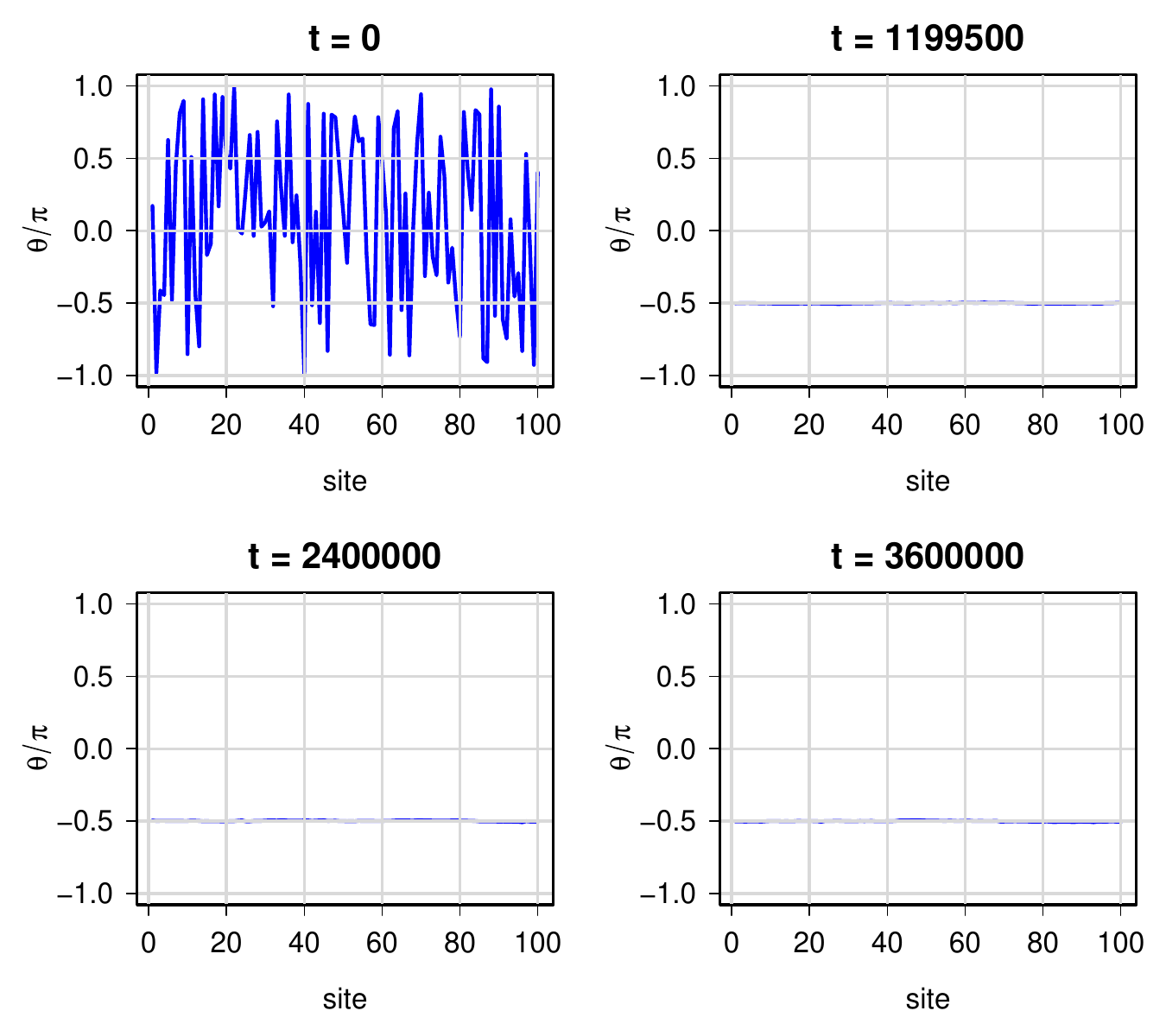}
    \caption*{(a)}
    \label{fig:open_eps0002_a}
\end{subfigure}
\hspace{0.01\textwidth}
\begin{subfigure}[t]{0.515\textwidth}
    \centering
    \includegraphics[
        width=1.06\linewidth,
        trim=0.2cm 0.2cm 0.2cm 0.2cm,
        clip
    ]{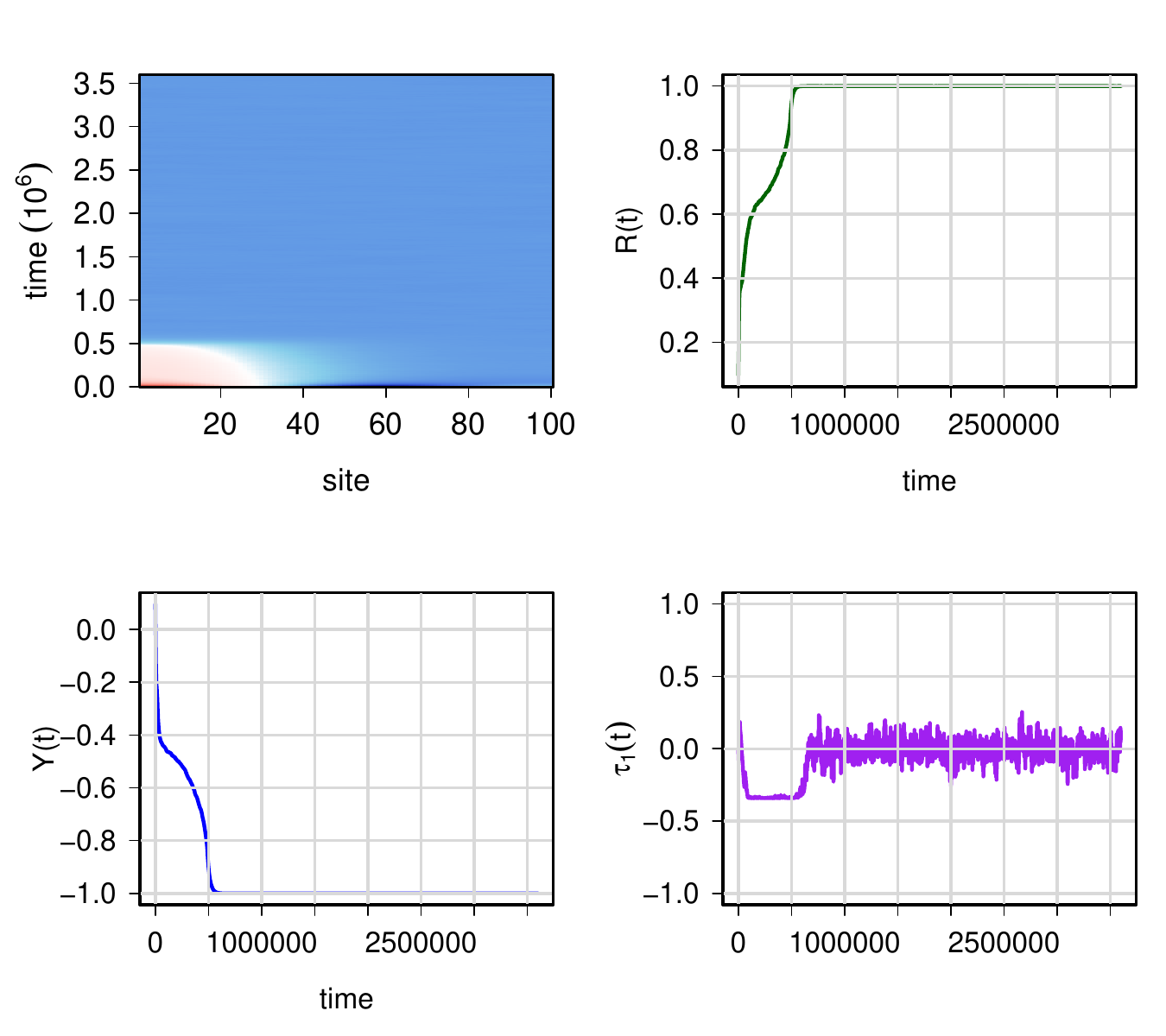}
    \caption*{(b)}
    \label{fig:open_eps0002_b}
\end{subfigure}
}

\vspace{-0.4em}

\caption{
ACCA dynamics with $N=100$ and $\varepsilon=0.002$ under open boundary conditions.
(a) Configuration snapshots at four different times.
(b) Heatmap of the time evolution together with the order parameters
$R(t)$, $Y(t)$ and $\tau_1(t)$.
}
\label{fig:100_open_eps0002}

\vspace{-0.6em}

\end{figure}

We now investigate the generalized dynamics of \texttt{Algorithm~\ref{alg:burst-acca}}
as a function of the burst parameters $k_{\mathrm{mid}}\in\{1,10,20,40\}$ and $k_{\mathrm{noise}}\in\{0,1,10,20,40\}$,
which control the degree of parallel midpoint averaging and the spatial extent of the stochastic perturbation.
Figure~\ref{fig:burst_summary} shows the late-time averages of the observables
$R$, $|Y|$, and $|\tau_1|$ as functions of $(k_{\mathrm{mid}},k_{\mathrm{noise}})$. The two parameters play qualitatively different roles.
Under open boundary conditions, increasing $k_{\mathrm{mid}}$ stabilizes the dynamics,
while increasing $k_{\mathrm{noise}}$ destabilizes it.
This is evident in panel~(a), where larger $k_{\mathrm{mid}}$ leads to higher values of $R$,
whereas larger $k_{\mathrm{noise}}$ suppresses alignment.
Panel~(b) shows that $|Y|$ behaves similarly:
parallel midpoint updates reinforce the selection of one of the two spin--flop branches,
while spatially extended noise weakens this macroscopic orientation. Under periodic boundary conditions the behavior is more subtle.
Panel~(c) shows that the dependence of $|\tau_1|$ on $k_{\mathrm{mid}}$ is generally non-monotone.
Parallel midpoint updates have a twofold effect:
they smooth local fluctuations and can help maintain a coherent twisted profile,
but when applied to many edges simultaneously they also induce more global rearrangements,
which may destabilize a given winding sector.
In contrast, increasing $k_{\mathrm{noise}}$ has a consistently destabilizing effect,
as perturbations applied to many sites erode the topological trapping mechanism.
Overall, the open system exhibits a largely monotone competition between stabilization and noise,
whereas the periodic system shows a more intricate interplay between parallel averaging and winding preservation.

\section{Discussion and future developments}
\label{sec:discussion}

The dynamics of the noisy \texttt{ACCA} model is governed by two distinct mechanisms. 
Under periodic boundary conditions, the winding number induces twisted configurations that confine the evolution near topological sectors, leading to trapping. 
Under open boundary conditions, this constraint is absent, and bi-modal noise drives the system toward two preferred orientations, producing a robust spin--flop regime.

Parallel burst updates clarify the interplay between these effects. 
Increasing the number of midpoint interactions enhances local averaging and stabilizes coherent structures, while increasing noisy updates destabilizes them. 
In open systems, these effects combine monotonically, whereas under periodic boundary conditions their interaction is non-monotone, reflecting the competition between averaging and topological constraints.

To summarize, the dynamics results from the interplay between local averaging, topology, and symmetry-breaking noise, which determines whether the system exhibits consensus, trapping, or persistent switching.

Future work includes quantifying transition times between winding sectors and studying structured or spatially correlated perturbations, bridging toward more realistic models of collective dynamics.
\begin{figure}[!htbp]
\centering

\begin{subfigure}[t]{\textwidth}
    \centering
    \includegraphics[width=1.0\textwidth]{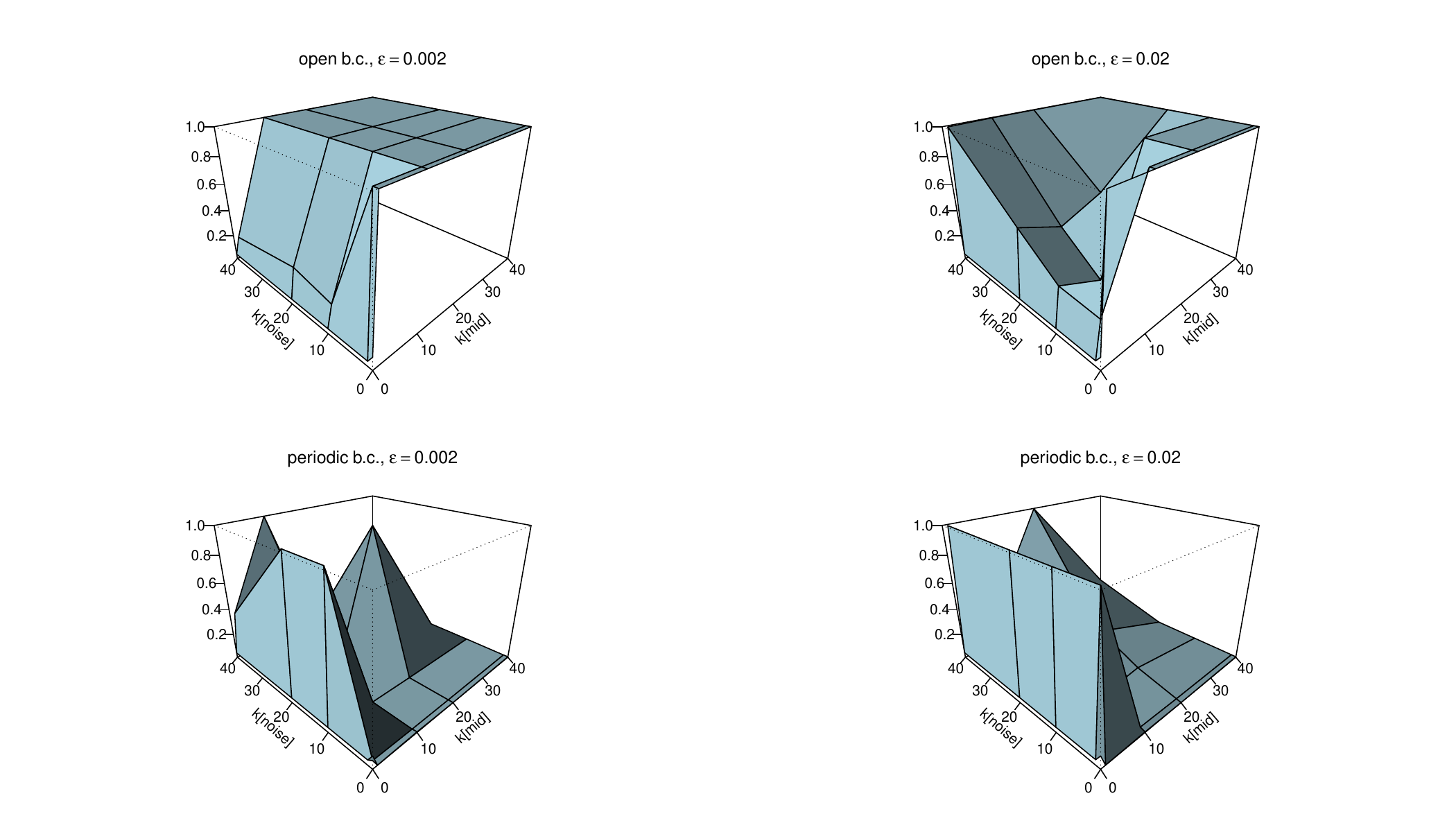}
    \caption{$R$}
\end{subfigure}
\vspace{-1.2em}
\begin{subfigure}[t]{\textwidth}
    \centering
    \includegraphics[width=1.0\textwidth]{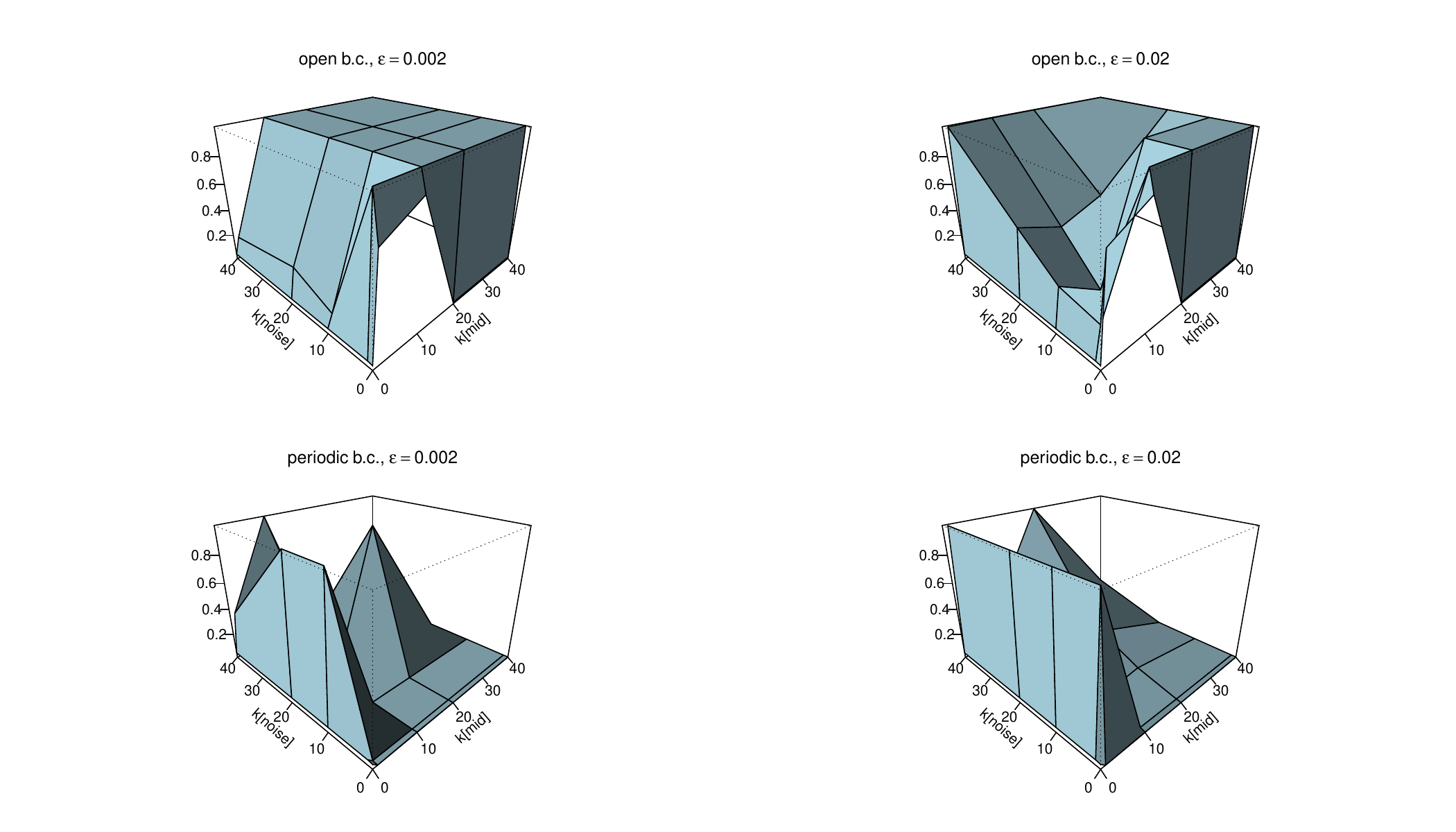}
    \caption{$|Y|$}
\end{subfigure}

\vspace{-1.2em}

\begin{subfigure}[t]{\textwidth}
    \centering
    \includegraphics[width=1.0\textwidth]{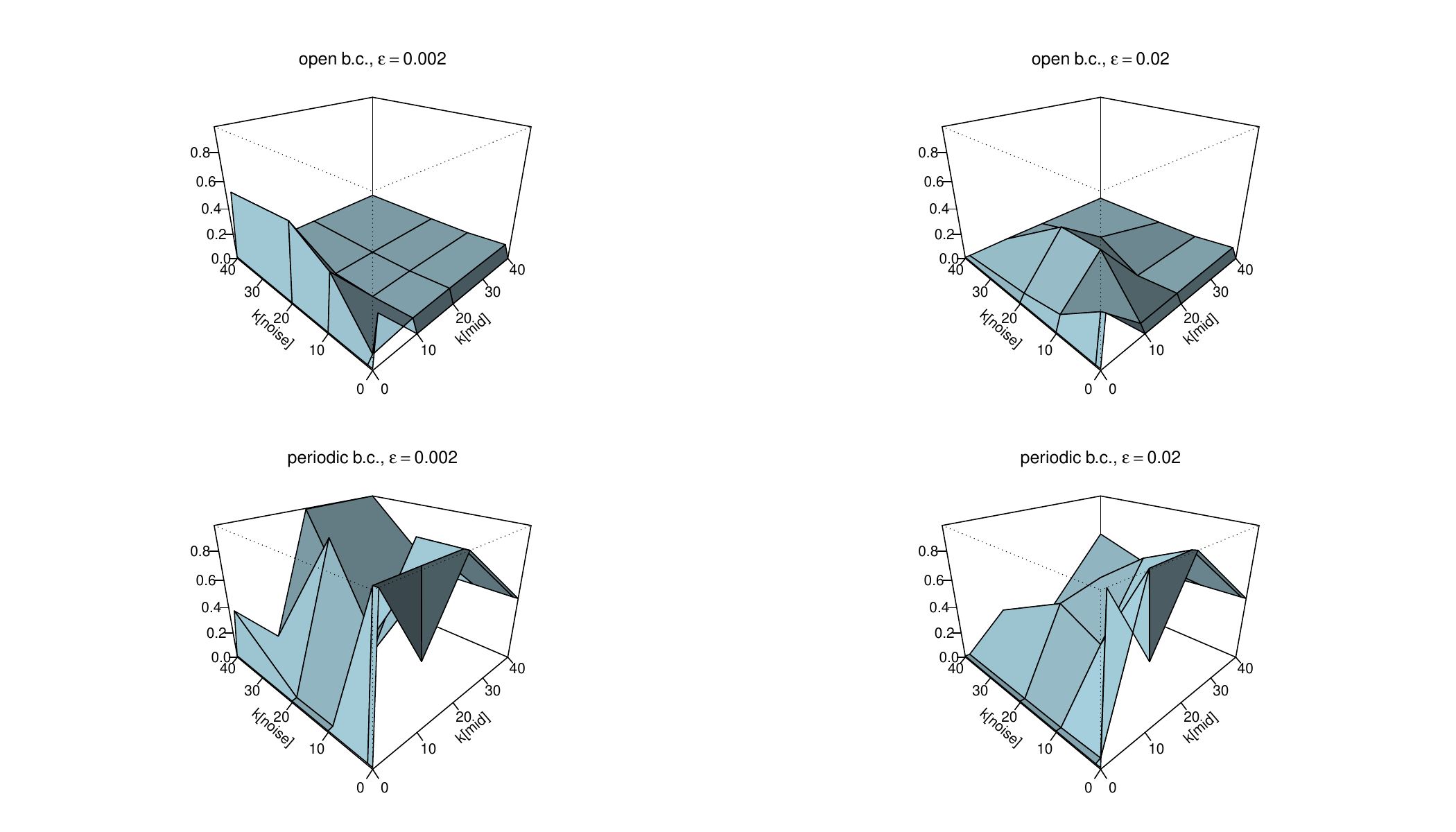}
    \caption{$|\tau_1|$}
\end{subfigure}

\vspace{-1.2em}

\caption{Late-time values of $R$ (panel (a)), $|Y|$ (panel ((b)), and $|\tau_1|$ (panel (c)) for
\texttt{Algorithm~\ref{alg:burst-acca}} as functions of $(k_{\mathrm{mid}},k_{\mathrm{noise}})$,
under open and periodic boundary conditions ($N=100$, $\varepsilon\in\{0.002,0.02\}$). For clarity, the axes correspond to increasing values of $k_{\mathrm{mid}}$ and $k_{\mathrm{noise}}$ starting from the origin.}
\label{fig:burst_summary}
\end{figure}
\FloatBarrier
\bibliographystyle{splncs04}
\bibliography{main}

\end{document}